\documentclass[12pt]{amsart}

\newtheorem{theorem}{Theorem}[section]
\newtheorem{lemma}[theorem]{Lemma}
\newtheorem{proposition}[theorem]{Proposition}
\newtheorem{corollary}[theorem]{Corollary}
\theoremstyle{definition}
\newtheorem{definition}[theorem]{Definition}
\newtheorem{example}[theorem]{Example}

\theoremstyle{remark}
\newtheorem{remark}[theorem]{Remark}

\newcommand{\Z}{\mathbb{Z}}
\newcommand{\N}{\mathbb{N}}
\newcommand{\M}{{\rm M}}

\newcommand{\eg}{\overset{(*)}{=}}

\newcommand{\mat}[1]{\left[\begin{matrix} #1 \end{matrix}\right]}

\usepackage{amssymb}
\usepackage{amsfonts}
\usepackage{latexsym}
\usepackage{amscd}
\usepackage[mathscr]{euscript}
\usepackage{xy} \xyoption{all}

\begin{document}

\title{Leavitt path algebras of Cayley graphs $C_n^j$}

\author{Gene Abrams}
\address{Department of Mathematics \\
University of Colorado \\
Colorado Springs, CO  80918, USA}
\email{abrams@math.uccs.edu}

\author{Stefan Erickson}
\address{Department of Mathematics \\
The Colorado College \\
Colorado Springs, CO  80903, USA}
\email{Stefan.Erickson@ColoradoCollege.edu }

\author{Crist\'{o}bal Gil Canto}
\address{Departamento de \'{A}lgebra, Geometr\'{\i}a y Topolog\'{\i}a\\
Universidad de M\'{a}laga\\
29071, M\'{a}laga, Spain}
\email{cristogilcanto@gmail.com / cgilc@uma.es}



\begin{abstract} Let $n$ be a positive integer.  For each $0\leq j \leq n-1$ we let $C_n^j$ denote the Cayley graph of the cyclic group $\mathbb{Z}_n$ with respect to the subset $\{1,j\}$.
Utilizing the Smith Normal Form process, we give an explicit description of the Grothendieck group of each of the Leavitt path algebras $L_K(C_n^j)$ for any field $K$.  Our general method significantly streamlines the approach that was used in previous work to establish this description in the specific case  $j=2$.  Along the way, we give necessary and sufficient conditions on the pairs $(j,n)$ which yield that this group is infinite.  We subsequently  focus on the case $j = 3$, where the structure of this group
turns out to be related to a Fibonacci-like sequence, called the Narayana's Cows sequence.
\end{abstract}

\maketitle

For any finite group $G$, and any set of generators $S$ of $G$, the {\it Cayley graph of } $G$ {\it with respect to} $S$ is the directed graph $E_{G,S}$ having vertex set $\{v_g \ | \ g\in G\}$, and in which there is an edge from $v_g$ to $v_h$ in case there exists (a necessarily unique) $s\in S$ with $h = gs$ in $G$.  For a positive integer  $n$  ($n\geq 3$)  and $j$ an integer with  $0 \leq j \leq n-1$, we denote by $C_n^j$  the Cayley graph $E_{G,S}$ of the cyclic group $G = \Z_n$ with respect to the generating subset $S = \{1,j\}$ of $\Z_n$. Concretely, the {\it Cayley graph} $C_n^j$ is the directed graph consisting of $n$ vertices $\{v_1, v_2, \hdots, v_n\}$  and $2n$ edges $\{e_1, e_2, \hdots, e_n, f_1, f_2, \hdots, f_n\}$ for which $s(e_i) = v_i,  \  r(e_i) = v_{i+1},  \  s(f_i) = v_i,   $ and $ r(f_i) = v_{i+j},$
where indices are interpreted ${\rm mod } \ n$. Intuitively, $C_n^j$ is the graph with $n$ vertices and $2n$ edges, where each vertex $v$ emits two edges, one to its ``next" neighboring vertex (which we will draw throughout this article in a clockwise direction), and one to the vertex $j$ places clockwise from $v$. For instance:

$$ C_4^0 = \ \ \ \ {
\def\labelstyle{\displaystyle}
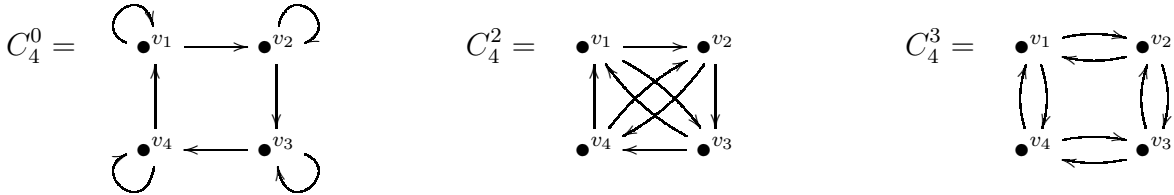
\xymatrix{ \bullet^{v_1} \ar[r]  \ar@(l,u) &  \bullet^{v_2} \ar@(u,r) \ar[d] \\  \bullet^{v_4} \ar@(d,l) \ar[u] & \bullet^{v_3} \ar[l] \ar@(r,d)}} \hskip2cm\
C_4^2 = \ \  {
\def\labelstyle{\displaystyle}
\xymatrix{ \bullet^{v_1} \ar@/^{5pt}/ [dr] \ar[r]   &  \bullet^{v_2} \ar@/^{5pt}/ [dl] \ar[d] \\  \bullet^{v_4}  \ar@/^{5pt}/ [ur] \ar[u] & \bullet^{v_3} \ar@/^{5pt}/ [ul] \ar[l]}}
\hskip2cm  \
C_4^3 =  \ \ {
\def\labelstyle{\displaystyle}
\xymatrix{ \bullet^{v_1} \ar@/^{5pt}/ [r] \ar@/^{5pt}/ [d]   &  \bullet^{v_2} \ar@/^{5pt}/ [d] \ar@/^{5pt}/ [l] \\  \bullet^{v_4}  \ar@/^{5pt}/ [u] \ar@/^{5pt}/ [r] & \bullet^{v_3} \ar@/^{5pt}/ [l] \ar@/^{5pt}/ [u]}}
$$
\vskip0.5cm

Leavitt path algebras of Cayley graphs were initially studied in \cite{ASch}, where a description is given of the Leavitt path algebras of the form  $L_K(C_n^{n-1})$. It is shown in \cite[Theorem 8]{ASch} that there are exactly four isomorphism classes represented by the collection $\{ L_K(C_n^{n-1}) \ | \ n\in \mathbb{N}\}$. In subsequent work, \cite[Section 2]{AA} contains the computation  of the two important (and closely related) integers $|K_0(L_K(C_n^j))|$ and ${\rm det}(I_n - A_{C_n^j}^t)$, where $A_{ ( -  ) }$ denotes the incidence matrix of a directed graph,   $A_{ ( -  ) }^t$ its transpose, and $K_0 ( - )$ denotes the Grothendieck group of a ring (see below for further details).     Specifically, it is shown that, for fixed $j$, the integers $|K_0(L_K(C_n^j))|$ are (up to sign) the entries in the ``$j^{{\rm th}}$  Haselgrove sequence" (\cite[Definition 2.2]{AA}), a sequence investigated in the 1940s by Haselgrove in \cite{H} (in part with an eye towards establishing Fermat's Last Theorem).

Also in \cite{AA}, the following collections of $K$-algebras are completely described up to isomorphism:

\medskip

\quad  -  $\{ L_K(C_n^0) \ | \ n\in \mathbb{N}\}$

\qquad \ \ \  there is only one such algebra,

 \smallskip

\quad  -  $\{ L_K(C_n^1) \ | \ n\in \mathbb{N}\}$

\qquad  \ \ \ there are infinitely many such algebras, each of them subsequently realized in terms of matrices over the standard Leavitt algebras $L_K(1,m)$, and

\smallskip


\quad  -  $\{ L_K(C_n^2) \ | \ n\in \mathbb{N}\}$

\qquad \ \ \ there are infinitely many such algebras, each of them subsequently realized up to isomorphism as the Leavitt path algebra of a graph having at most three vertices; moreover, the groups $K_0(L_K(C_n^2))$ are described explicitly in terms of the integers appearing in the second Haselgrove sequence $H_2(n)$, together with integers arising from the standard Fibonacci sequence.

 \medskip

 The descriptions of all these algebras follow from an application of the powerful tool known as  the Restricted Algebraic Kirchberg-Phillips Theorem (see  Section \ref{Sect:K_0lpasC_n^3}).

In the current article we continue the investigation of Leavitt path algebras associated to Cayley graphs.  
 We begin in Section \ref{Background} by giving  the requisite background information about the topic. In Section \ref{Section:SmithnfC_n^3} we present a method to compute  the Grothendieck group of the Leavitt path algebra $L_K(C_n^j)$ for any $n\geq 3$ and $0\leq j\leq n-1$.  The  approach we utilize is completely   different from the work done in \cite{AA} for the case $C_n^2$.  Specifically, in Theorem \ref{Thm:SNF}, we show how  to reduce the computation of the Smith Normal Form of the $n \times n$ matrix $I_n - A_{C_n^j}^t$ (data which provides a complete description of the $K_0$ group) to that  of calculating the Smith Normal Form of the $j \times j$ matrix $(M_j^n)^t - I_j$, where $M_j$ is the companion matrix  of a polynomial associated to a recursive sequence which arises naturally in this context.   With this result in hand, we then use the Determinant Divisors Theorem to achieve the desired description of the finitely generated abelian group $K_0(L_K(C_n^j))$.

With the  general result Theorem \ref{Thm:SNF} in hand,  in Section \ref{Sectionjequals2and3}  we focus our attention on the specific case $j=3$.
This leads us to a number-theoretic analysis which involves some perhaps surprising and apparently nontrivial connections between the third Haselgrove sequence $H_3$ and the Narayana's Cows sequence $G$ (Proposition \ref{Cor:alpha_3=H_3}). Subsequently, we briefly  mention how the result for $j=2$ presented in
\cite{AA} (which was achieved by using an extremely ad hoc proof that spans more than a dozen journal pages) may be re-established with only a modicum of work.

In the final section we show how Theorem \ref{Thm:SNF}, together with the Restricted Algebraic Kirchberg-Phillips Theorem,  allows us to realize any  Leavitt path algebra of the form  $L_K(C_n^j)$ (specifically, the Leavitt path algebra of a graph with $n$ vertices) up to isomorphism as the Leavitt path algebra of a graph with at most $j+1$ vertices.

\section{Background information}\label{Background}

For any field $K$ and directed graph $E$ the Leavitt path algebra $L_K(E)$ has been the focus of sustained investigation since 2004.  We put off until Section \ref{Sect:K_0lpasC_n^3}  a detailed description of these algebras, choosing  instead to  focus first on the construction of a monoid which can be carried out for any directed graph.    
For a directed graph $E$ having $n$ vertices $v_1, v_2,Ê\dots, v_n$  we denote by $A_E$ the usual {\it incidence matrix} of $E$, namely, the matrix $A_E = (a_{i,j})$ where, for each pair $1\leq i,j \leq n$, the entry $a_{i,j}$ denotes the number of edges $e$ in $E$ for which $s(e)=v_i$ and $r(e) = v_j$. Let $E$ be a directed graph  with vertices $v_1, v_2, \hdots, v_n$ and incidence matrix $A_E = (a_{i,j})$.  We let $F_n$ denote the free abelian monoid on the generators $v_1, v_2, \hdots, v_n$ (so $F_n \cong \bigoplus_{i=1}^n \Z^+$ as monoids).  We denote the identity element of this monoid by $z$.    We define a relation $\approx$ on $F_n$ by setting
$$v_i  \approx  \sum_{j=1}^n a_{i,j}v_j$$
for each non-sink $v_i$, and  denote by $\sim_E$ the equivalence relation on $F_n$ generated by $\approx$.

For two $\sim_E$ equivalence classes $[x]$ and $[y]$ we define $[x] + [y] = [x+y]$; it is straightforward to show that this gives a well-defined associative binary operation on the set of $\sim_E$ equivalence classes, and that $[z]$ acts as an identity element for  this operation.  We denote the resulting
 {\it graph monoid}  $F_n / \sim_E$  by $M_E.$    Specifially,

\medskip
{\bf Definition.}
For any $n\geq 1$ and $0\leq j \leq n-1$, the graph monoid  $M_{C_n^j}$ is the free abelian monoid $F_n$ generated by $[v_1], [v_2], \hdots, [v_n]$, subject to the relations $$[v_i ]= [v_{i+1}] + [v_{i+j}]$$
(for all $1\leq i \leq n$), where subscripts are interpreted ${\rm mod} \ n$.

\medskip

For examples of the graph monoid $M_{C_n^j}$, see both  \cite[p. 3]{ASch} (in which the graph monoid  $M_{C_3^2}$ is shown to consist of the five elements $\{[z], \ [v_1], \  [v_2], \ [v_3], \ [v_1]+[v_2]+[v_3] \}$) and  \cite[pp. 3, 4]{AA} (where  the graph monoid $M_{C_4^2}$ associated to the graph $C_4^2$ is explicitly described).

We present now a streamlined version of the germane background information which will be utilized throughout the remainder of the article.  Much of this discussion appears in \cite{AA}.

For a unital $K$-algebra $A$,  the set of isomorphism classes of finitely generated projective left $A$-modules is denoted by $\mathcal{V}(A)$.  We denote the elements of  $\mathcal{V}(A)$ using brackets; for example, $[A] \in \mathcal{V}(A)$ represents the isomorphism class of the left regular module ${}_AA$.   $\mathcal{V}(A)$ is a monoid, with operation $\oplus$, and zero element $[\{0\}]$.   The monoid $(\mathcal{V}(A), \oplus)$ is {\it conical}: this means that the sum of any two nonzero elements of $\mathcal{V}(A)$ is nonzero, or, rephrased, that $\mathcal{V}(A)^* = \mathcal{V}(A) \setminus \{0\}$ is a semigroup under $\oplus$. The following striking property of Leavitt path algebras  was established in \cite[Theorem 3.5]{AMP}.
$$  \mathcal{V}(L_K(E)) \cong M_E \ \mbox{as monoids}. \ \ \ \ \
\mbox{Moreover, } \   [L_K(E)] \leftrightarrow \sum_{v\in E^0} [v]  \ \mbox{under this isomorphism.}$$

A unital $K$-algebra $A$ is called {\it purely infinite simple} in case $A$ is not a division ring, and $A$ has the property that for every nonzero element $x$ of  $A$ there exist $b,c\in A$ for which $bxc=1_A$.     It is shown in \cite[Corollary 2.2]{AGP} that if $A$ is  a unital purely infinite simple $K$-algebra, then the semigroup $(\mathcal{V}(A)^*, \oplus)$ is in fact a group, and, moreover, that $\mathcal{V}(A)^* \cong K_0(A)$, the Grothendieck group of $A$.     
Summarizing: when $L_K(E)$ is unital purely infinite simple we have the following isomorphisms of groups:
$$ K_0(L_K(E)) \cong \mathcal{V}(L_K(E))^* \cong M_E^*.$$

In particular, in this situation we have   $|K_0(L_K(E))| = | M_E^* |$.
The finite graphs $E$ for which the Leavitt path algebra $L_K(E)$ is purely infinite simple have been explicitly described (using only graph-theoretic properties of $E$)  in \cite{AAP2}.
 This result, together with the preceding discussion, immediately yields

\medskip
{\bf Proposition.} \cite[Proposition 1.3]{AA} For each $n\geq 1$, and for each $0\leq j \leq n-1$,  the $K$-algebra $L_K(C_n^j)$ is unital purely infinite simple.  In particular, $M_{C_n^j}^* = (M_{C_n^j} \setminus \{[z]\},+)$ is a group, necessarily isomorphic to $K_0(L_K(C_n^j))$.

\medskip
{\it Our goal here}  is to describe the group  $M_{C_n^j}^*$.   Our motivation for doing so is twofold.  First, the description turns out to be inherently interesting in its own right (involving the aforementioned Haselgrove sequences).  Second, we may utilize the  structure of this group, viewed as $ K_0(L_K(C_n^j))$,  to (quite surprisingly) glean information about the algebra $L_K(C_n^j)$ itself.   This is done by invoking the so-called Restricted Algebraic Kirchberg-Phillips Theorem, which we describe fully  in Section \ref{Sect:K_0lpasC_n^3}.

We now review  a useful  computational tool. Let $M \in \M_n(\mathbb{Z})$ and view $M$ as a linear transformation $M:\mathbb{Z}^n \rightarrow \mathbb{Z}^n$ via left multiplication on columns.    The cokernel of $M$
is a finitely generated abelian group, having at most $n$ cyclic direct summands; as such, by the invariant factors version of the Fundamental Theorem of Finitely Generated Abelian groups, we have
$${\rm Coker}(M)\cong \mathbb{Z}_{s_\ell} \oplus \mathbb{Z}_{s_{\ell + 1}} \oplus \cdots \oplus \mathbb{Z}_{s_n}$$
for some $1\leq \ell \leq n$, where either $n=\ell$ and $s_n=1$ (i.e., ${\rm Coker}(M)$ is the trivial group), or there are (necessarily unique) nonnegative  integers $s_\ell, s_{\ell + 1}, \dots , s_n$, for which the nonzero values $s_\ell, s_{\ell + 1}, \dots , s_r$ satisfy $s_i \geq 2$ and $s_i \vert s_{i+1}$ for $ \ell \leq i \leq r-1$, and $s_{r+1} = \cdots = s_n = 0$.
 ${\rm Coker}(M)$ is a finite group if and only if  $r=n$ (i.e., there are no zeros in the sequence $s_\ell, \dots, s_n$).
In case $\ell > 1$, we define $s_1 = s_2 = \cdots = s_{\ell-1} = 1$.    Clearly then we have
$${\rm Coker}(M)\cong \mathbb{Z}_{s_1} \oplus \mathbb{Z}_{s_{2}} \oplus \cdots \oplus \mathbb{Z}_{s_{\ell}} \oplus  \cdots  \oplus \mathbb{Z}_{s_n},$$
since any additional direct summands are isomorphic to the trivial group $\Z_1$.

\begin{remark}\label{isomorphicCokerremark}
{\rm
It is straightforward to establish that  if $P,Q$ are invertible in $\M_n(\mathbb{Z})$, then ${\rm Coker}(M)\cong {\rm Coker}(PMQ)$.  (We note that any such $P$ and $Q$ must have determinant $\pm 1$.)  This means that if $N \in \M_n(\mathbb{Z})$ is a matrix which is constructed by performing any sequence of $\mathbb{Z}$-elementary row and/or column operations starting with $M$, then ${\rm Coker}(M)\cong {\rm Coker}(N)$ as abelian groups.   }
\end{remark}

\begin{definition}\label{Def:Smithnormalform}{\rm
Let $M \in \M_n(\mathbb{Z})$, and suppose ${\rm Coker}(M)\cong \mathbb{Z}_{s_1}\oplus\mathbb{Z}_{s_2}\oplus \cdots \oplus \mathbb{Z}_{s_n}$ as described above.  The \emph{Smith Normal Form} of $M$  is the $n\times n$ diagonal matrix $ {\rm diag}(s_1, s_2, \ldots, s_r, 0, \dots, 0)$.
}
\end{definition}

For any matrix $M \in \M_n(\mathbb{Z})$, the Smith Normal Form of $M$ exists and is unique. If $D \in \M_n(\mathbb{Z})$ is a diagonal matrix with entries $d_1,d_2,\ldots,d_n$ then
clearly ${\rm Coker}(D)\cong \mathbb{Z}_{d_1}\oplus\mathbb{Z}_{d_2}\oplus \cdots \oplus \mathbb{Z}_{d_n}$.
In the end we have the following.

\begin{proposition}\label{Prop:Smithnormalform} Let $M \in \M_n(\mathbb{Z})$, and let $S$ denote the Smith Normal Form of $M$. Suppose the diagonal entries of $S$ are $s_1,s_2,\ldots,s_n$. Then
$${\rm Coker}(M)\cong \mathbb{Z}_{s_1}\oplus\mathbb{Z}_{s_2}\oplus \cdots \oplus \mathbb{Z}_{s_n}.$$
In particular, if there are no zero entries in the Smith Normal Form of $M$, then $|{\rm Coker}(M)| = s_1 s_2 \cdots s_n = |{\rm det}(S)| = |{\rm det}(M)|$.
\end{proposition}

The key computational device we will utilize to compute the Smith Normal Form of a matrix $M$ is the following.

\medskip

{\bf Determinant Divisors Theorem.} \cite[Theorem II.9]{N} Let $M \in \M_n(\Z)$.
Define $\alpha_0 := 1$, and, for each  $1\leq i \leq n$, define the  \emph{$i^{th}$ determinant divisor of} $M$ to be the integer
$$\alpha_i := \text{ the greatest common divisor of  the set of all } i \times i \text{ minors of } M.$$
Let $s_1, s_2, ... , s_n$ denote the diagonal entries of the Smith Normal Form of $M$, and assume that each $s_i$ is nonzero.    Then   $$s_i = \frac{\alpha_{i}}{\alpha_{i-1}}$$ for each $1 \leq i \leq n$.

\medskip

Suppose now that $E$ is a finite directed graph having $n$ vertices $v_1, v_2, \dots, v_n$.
Consider the matrix $I_n - A_E^t$, where $I_n$ is the identity $n \times n$ matrix.
As above, we may view
$I_n - A_E^t$ as a linear transormation $\Z^n \rightarrow \Z^n$.  Invoking the discussion in \cite[Section 3]{AALP},   in the case $L_K(E)$ is purely infinite simple
we have that  $$K_0(L_K(E)) \cong M_E^* \cong
{\rm Coker}(I_n - A_E^t).$$
Under this isomorphism   $[v_i ] \mapsto \vec{b_i} + {\rm Im }(I_n - A_E^t)$, where  $\vec{b_i}$ is the element of $\Z^n$ which is $1$ in the $i^{th}$ coordinate and $0$ elsewhere. In other words, when $L_K(E)$ is purely infinite simple, then $K_0(L_K(E))$ is the cokernel of the linear transformation $I_n - A_E^t: \Z^n \rightarrow \Z^n$ induced by matrix multiplication.

\begin{example} Suppose $E=R_m$ ($m \geq 2$), the ``rose with $m$ petals" graph having one vertex and $m$ loops.  Because $m\geq 2$, $L_K(R_m)$ is purely infinite simple.   Here  $A_E$ is the $1 \times 1$ matrix $(m)$, so $I_1 -A_{R_m}^t$ is the $1 \times 1$ matrix $(1-m)$, and we have
$$K_0(L_K(R_m))\cong \Z^1_{1-m} \cong \Z_{m-1}.$$
\end{example}

  Proposition \ref{Prop:Smithnormalform} then yields the following.

\begin{proposition}\label{Prop:K_0andSNF} Suppose $E$ is a finite graph with $|E^0|=n$,  and suppose also that $L_K(E)$ is purely infinite simple. Let $S$ be the Smith Normal Form of the matrix $I_n-A_E^t$, with diagonal entries $s_1,s_2,\ldots,s_n$. Then
$$K_0(L_K(E))\cong \mathbb{Z}_{s_1}\oplus\mathbb{Z}_{s_2}\oplus \cdots \oplus \mathbb{Z}_{s_n}.$$
\end{proposition}

Moreover, if $K_0(L_K(E))$ is finite, then an analysis of the Smith Normal Form of the matrix $I_n - A_E^t$ yields
$$|  K_0(L_K(E)) | = | {\rm det}(I_n - A_E^t) |.$$
(This is immediate, since as noted above any invertible matrix in $\M_n(\Z)$ has determinant $\pm 1$.)   Conversely, $K_0(L_K(E))$ is infinite if and only if ${\rm det}(I_n - A_E^t) = 0.$

In \cite{H}, Haselgrove introduces for each pair of positive integers $n,k$ the integer
$$
a_k(n) :=   \prod_{\ell = 0}^{n-1} (1 - \omega_{\ell} - \omega_{\ell}^k),
$$
where $\omega_\ell = \cos( \frac{2 \pi \ell}{n}) + i \sin (\frac{2 \pi \ell}{n})$   in $\mathbb{C}$.  (That this expression indeed yields an integer follows from some elementary number theory.)    Subsequently, in  \cite[Definition 2.2]{AA}  the integer
$$|a_k(n)| \ \mbox{is denoted by } H_k(n),$$
and, for fixed $k$, the sequence
$$H_k(1), H_k(2), H_k(3), ...$$  is referred to as the {\it $k^{th}$ Haselgrove sequence}.
 It is of historical interest to note that Haselgrove's motivation for considering these integers $a_k(n)$   was for their potential use in establishing a connection between a resolution of Fermat's Last Theorem (at the time, of course, Fermat's Last {\it Conjecture})  and some integers which share properties of the Mersenne numbers.

We recall some properties of the integers $a_k(n)$ and $H_k(n)$ (established  \cite{AA}) which show why these are germane in the current context,  and then finish the section with a new result which will be useful in the sequel.


\medskip
{\bf Proposition.}  (See  \cite[Section 2]{AA})
Let $n\in \mathbb{N}$ and $0\leq k \leq n-1$.
\begin{enumerate}
\item   $a_k(n) = \det(I_n-A_{C_n^k}^t)$.   (This is established using  the notion of {\it circulant} matrices.)

\item  $ \det(I_n-A_{C_n^k}^t) \leq 0.$   (This is established by some elementary analysis in $\mathbb{C}$.)
\newline
 \hspace{.5in}  In particular, $\det(I_n-A_{C_n^k}^t) = -H_k(n)$.

\item     If $H_k(n)>0$, then $ |K_0(L_K(C_n^k))| = H_k(n) = | {\rm Coker}(I_n-A_{C_n^k}^t)|.$

\item      $H_k(n)=0$ if and only if  $K_0(L_K(C_n^k))$ is infinite.

\end{enumerate}



\begin{proposition}\label{Lem:H_k(n)=0}
Let $n \in \N$ and $0\leq k \leq n-1$.   Then $H_k(n) = 0$ if and only if $k \equiv 5 \pmod{6}$ and $n \equiv 0 \pmod{6}$.
\end{proposition}

\begin{proof}
By the previous recollections from \cite{AA} we have
\[
H_k(n) = -\prod_{\ell = 0}^{n-1} (1 - \omega_{\ell} - \omega_{\ell}^k),
\]
where $\omega_\ell = \cos( \frac{2 \pi \ell}{n}) + i \sin (\frac{2 \pi \ell}{n})$   in $\mathbb{C}$.  Then $H_k(n) = 0$ if and only if one of the factors $1 - \omega_{\ell} - \omega_{\ell}^k = 0$ for some $0 \leq \ell \leq n-1$.
In particular, we must have $\omega_{\ell} + \omega_{\ell}^k = 1$.
Letting $\theta = \frac{2 \pi \ell}{n}$ (which we may assume $0 \leq \theta < 2 \pi$), we have
$$\cos(\theta) + \cos(k \theta) = 1 \ \ \ \  \mbox{and} \ \ \ \
\sin(\theta) + \sin(k \theta) = 0. $$
The second equation implies that $k \theta \equiv -\theta \pmod {2 \pi}$ or that $k \theta \equiv 2\pi - \theta \pmod{2 \pi}$. In the first case,
\[
\cos(\theta) + \cos(k \theta) = \cos(\theta) + \cos(-\theta) = 0,
\]
which is a contradiction. Hence, the second condition must be true, and
\[
\cos(\theta) + \cos(k \theta)
= \cos(\theta) + \cos(2\pi-\theta)
= 2 \cos(\theta) = 1.
\]
Hence, $\theta = \frac{\pi}{3}$ or $\theta = \frac{5\pi}{3}$.
Substituting in $\theta = \frac{2 \pi \ell}{n}$, we see that
\[
\frac{2 \pi \ell}{n} = \frac{\pi}{3} \implies n = 6\ell, \ \ \ \  \mbox{or} \ \  \ \
\frac{2 \pi \ell}{n} = \frac{5\pi}{3} \implies 5n = 6\ell.
\]
In either case, $n \equiv 0 \pmod{6}$.

Similarly, $k \theta \equiv 2\pi - \theta \pmod{2\pi}$ implies that
$(k+1)\theta \equiv 2 \pi \equiv 0 \pmod{2\pi}$. Hence, for some integer $m$,
\[
(k+1) \frac{\pi}{3} = 2 m \pi \implies k+1 = 6 m, \ \ \ \ \mbox{or} \ \ \ \
(k+1) \frac{5\pi}{3} = 2 m \pi \implies 5(k+1) = 6 m.
\]
In either case, $k+1 \equiv 0 \pmod{6}$, or $k \equiv 5 \pmod{6}$.

Finally, when $n \equiv 0 \pmod{6}$ and $k\equiv 5 \pmod{6}$, then letting $\ell = \frac{n}{6}$ implies that
\[
1 - \omega_{\ell} - \omega_{\ell}^k
= 1 - e^{\frac{2 \pi i}{6}} - (e^{\frac{2 \pi i}{6}})^k
= 1 - e^{\frac{2 \pi i}{6}} - (e^{\frac{2 \pi i}{6}})^{-1}
= 0
\]
Since one of the factors of $H_k(n)$ is zero, we conclude that $H_k(n) = 0$.
\end{proof}

\medskip

\section{The Smith Normal Form of the matrix $I_n - A_{C_n^j}^t$}\label{Section:SmithnfC_n^3}

In order to investigate the Leavitt path algebras $\{L_K(C_n^j) \ | \ n\in \mathbb{N}\}$, we begin in a manner similar to that used in the case $C_n^2$ studied in \cite[Section 4]{AA}. The generating relations for $M_{C_n^j}^*$ are given by
$$[v_i] = [v_{i+1}] + [v_{i+j}] $$
for $1\leq i \leq n$, where subscripts are interpreted ${\rm mod }\ n$. We will focus on the element  $[v_1]$ of $M_{C_n^j}^*$;  corresponding to any statement  established for  $[v_1]$ in  $M_{C_n^j}^*$,  there will be (by the symmetry of the relations)  an analogous statement in  $M_{C_n^j}^*$ for each $[v_i]$, $1\leq i \leq n$.

The computation of the Smith Normal Form of the $n\times n$ matrix $I_n - A_{C_n^j}^t$ is the key tool for determining the $K_0$ of the Leavitt path algebra $C_n^j$.
We show below that this  computation  reduces to calculating the Smith Normal Form of a $j \times j$ matrix.
The authors are quite grateful to  M. Iovanov for suggesting this approach.




\begin{definition}{\rm   Let $
p(x) = x^j + c_{j-1} x^{j-1} + \dots + c_1 x + c_0
$ be a degree $j$ monic polynomial with integer coefficients.  The \emph{companion matrix} $M(p)$ of $p(x)$  is the $j \times j$ matrix
\[
M(p) :=
\mat{
0      & 0      & 0      & \dots  & 0      & -c_0     \\
1      & 0      & 0      & \dots  & 0      & -c_1     \\
0      & 1      & 0      & \dots  & 0      & -c_2     \\
\vdots & \vdots & \vdots & \ddots & \vdots & \vdots   \\
0      & 0      & 0      & \dots  & 0      & -c_{j-2} \\
0      & 0      & 0      & \dots  & 1      & -c_{j-1}
}.
\]
For  $j \geq 2$  we define    $p_j(x) = x^j - x^{j-1} - 1 \in \Z[x]$. The  companion matrix of $p_j(x)$, which we will denote by $M_j$, is then the $j \times j$ matrix
\[
M_j := M(p_j(x)) =
\mat{
0      & 0      & 0      & \dots  & 0      & 1      \\
1      & 0      & 0      & \dots  & 0      & 0      \\
0      & 1      & 0      & \dots  & 0      & 0      \\
\vdots & \vdots & \vdots & \ddots & \vdots & \vdots \\
0      & 0      & 0      & \dots  & 0      & 0      \\
0      & 0      & 0      & \dots  & 1      & 1
}.
\]
}
\end{definition}


%
%


\begin{remark}   Clearly the two matrices  $I_n - A_{C_n^j}^t$ and  $A_{C_n^j} - I_n$  have the same Smith Normal Form (i.e., have isomorphic cokernels).   In the sequel we choose to  analyze the latter, because it is easier to work with computationally.   A similar statement holds for the matrices $M_n^j - I_j$ and $(M_n^j)^t - I_j$.     We note that in a more general analysis of the structure of Leavitt path algebras than the one carried out here, such an interchange might  possibly forfeit some important information.
\end{remark}

\begin{theorem}\label{Thm:SNF}
Let $n \geq j$.  Then ${\rm Coker}(A_{C_n^j} - I_n) \cong  {\rm Coker}(M_j^{n} - I_j)$.
\end{theorem}

\begin{proof}
The proof in the case $j \leq n \leq 2j$ is quite similar to the proof which we give here for the case $n>2j$, but requires some extra computational and notational energy;
we thereby  leave the proof of the $j \leq n \leq 2j$ case  to the interested reader. Thus we assume that $n>2j$.
By definition, the entry in the $k^{\mathrm{th}}$ row and $\ell^{\mathrm{th}}$ column of $A_{C_n^j} - I_n$ is given by
\[
( A_{C_n^j} - I_n)_{k \ell} =
\begin{cases}
-1 & \text{ if $\ell = k$ } \\
 1 & \text{ if $\ell \equiv k+1 \text{ or } k+j \pmod{n}$ } \\
 0 & \text{ otherwise. }
\end{cases}
\]
Our goal is to obtain a diagonal matrix with integers along the diagonal through elementary row and column operations involving only integer multiples.

We first note that the bottom left $j \times j$ submatrix of the matrix $A_{C_n^j} - I_n$ can be written as $M_j$ with the columns cyclically permuted:
\[
\mat{
   1   & 0 & 0 & \dots  & 0 &    0   \\
   0   & 1 & 0 &        & 0 &    0   \\
   0   & 0 & 1 &        & 0 &    0   \\
\vdots &   &   & \ddots &   & \vdots \\
   0   & 0 & 0 & \dots  & 1 &    0   \\
   1   & 0 & 0 & \dots  & 0 &    1   \\
}
=
M_j P,
\]
where $P$ is the $j \times j$ permutation matrix
\[
P =
\mat{
   0   & 1 & 0 & \dots  & 0 &    0   \\
   0   & 0 & 1 &        & 0 &    0   \\
   0   & 0 & 0 &        & 0 &    0   \\
\vdots &   &   & \ddots &   & \vdots \\
   0   & 0 & 0 & \dots  & 0 &    1   \\
   1   & 0 & 0 & \dots  & 0 &    0   \\
}
\]

The first $(n-2j)$ reduction steps of the Smith Normal Form will result in an $(n-2j) \times (n-2j)$ identity submatrix in the upper left corner.
On the bottom $j$ rows, the $i^{\mathrm{th}}$ reduction step adds the $i^{\mathrm{th}}$ column to the $(i+1)^{\mathrm{th}}$ column and the $(i+j)^{\mathrm{th}}$ column, then zeroes out the $i^{\mathrm{th}}$ column.
The matrix that accomplishes this reduction step is
\[
\mat{
\mathbf{v}_i & \mathbf{v}_{i+1} & \dots & \mathbf{v}_{i+j-1} \\
}
\cdot
\mat{
   1   & 0 & 0 & \dots   & 0 &    1   \\
   1   & 0 & 0 &         & 0 &    0   \\
   0   & 1 & 0 &         & 0 &    0   \\
\vdots &   &   & \ddots  &   & \vdots \\
   0   & 0 & 0 &         & 0 &    0   \\
   0   & 0 & 0 & \dots   & 1 &    0   \\
}
=
\mat{
\mathbf{v}_i + \mathbf{v}_{i+1} & \dots & \mathbf{v}_{i+j-1} & \mathbf{v}_{i} \\
},
\]
and this matrix is conjugate to the companion matrix $M_j$ via the permutation matrix $P$:
\[
\mat{
   1   & 0 & 0 & \dots   & 0 &    1   \\
   1   & 0 & 0 &         & 0 &    0   \\
   0   & 1 & 0 &         & 0 &    0   \\
\vdots &   &   & \ddots  &   & \vdots \\
   0   & 0 & 0 &         & 0 &    0   \\
   0   & 0 & 0 & \dots   & 1 &    0   \\
}
= P^{-1} M_j P.
\]
After $i$ reduction steps, the first $j \times j$ submatrix with nonzero column vectors on the bottom $j$ rows will be
\[
M_j P \cdot (P^{-1} M_j P)^{i} = M_j^{i+1} P.
\]


Denote by $Q$ the $j \times j$ matrix
\[
Q =
\mat{
  -1   &  1 &  0 & \dots  & 0  &   0   \\
   0   & -1 &  1 &        & 0  &   0   \\
   0   &  0 & -1 &        & 0  &   0   \\
\vdots &    &    & \ddots &    & \vdots \\
   0   &  0 &  0 &        & -1 &   1   \\
   0   &  0 &  0 & \dots  & 0  &  -1   \\
}.
\]
Then after $n-2j$ reduction steps, we have
\[
A_{C_{n}^{j}} - I_n \sim
\mat{
I_{n-2j}            & 0_{(n-2j) \times j} &   0_{(n-2j) \times j}    \\
0_{j \times (n-2j)} &          Q          &          M_j P           \\
0_{j \times (n-2j)} &  M_j^{n-2j+1} P     &            Q             \\
}.
\]

Because each reduction step only adds previous columns to the following existing columns, the next $j$ reduction steps results in
\[
A_{C_{n}^{j}} - I_n \sim
\mat{
I_{n-j}            &  0_{(n-j) \times j}        \\
0_{j \times (n-j)} &  M_j^{n-j+1} P + Q \\
}.
\]

Finally, we adjust the bottom right $j \times j$ matrix by
adding the $(n-j+1)^{\mathrm{th}}$ column through the $(n-1)^{\mathrm{th}}$ column to the $n^{\mathrm{th}}$ column,
adding the $(n-j+1)^{\mathrm{th}}$ column through the $(n-2)^{\mathrm{th}}$ column to the $(n-1)^{\mathrm{th}}$ column, and so on until
adding the $(n-j+1)^{\mathrm{th}}$ column to the $(n-j+2)^{\mathrm{th}}$ column.
This procedure is equivalent to multiplying $M_j^{n-j+1} P + Q$ on the right by the $j \times j$ matrix
\[
R =
\mat{
   1   &  1 &  1 & \dots  & 1 &    1   \\
   0   &  1 &  1 &        & 1 &    1   \\
   0   &  0 &  1 &        & 1 &    1   \\
\vdots &    &    & \ddots &   & \vdots \\
   0   &  0 &  0 &        & 1 &    1   \\
   0   &  0 &  0 & \dots  & 0 &    1   \\
}.
\]
A straightforward calculation shows that
\[
P R = M_j^{j-1} \qquad \text{and} \qquad Q R = -I_j,
\]
resulting in the final bottom right $j \times j$ submatrix is
\[
(M_j^{n-j+1} P + Q) R = M_j^n - I_j.
\]
Therefore,
\[
A_{C_{n}^{j}} - I_n \sim
\mat{
      I_{n-j}      & 0_{(n-j) \times j} \\
0_{j \times (n-j)} &    M_j^n - I_j     \\
}
\]
and thus by Remark \ref{isomorphicCokerremark}, we conclude that ${\rm Coker}(A_{C_n^j} - I_{n}) \cong  {\rm Coker}(M_j^n - I_j)$.
\end{proof}



\section{The case $j=3$, and the case $j=2$ (briefly) revisited}\label{Sectionjequals2and3}

%
%
%

We have shown in Theorem \ref{Thm:SNF}  that for any $n \geq j$, the cokernel of the $n \times n$ matrix $A_{C_n^j} - I_n$ is isomorphic to the cokernel of the $j\times j$ matrix $M_j^n - I_j$.
In this section we investigate in detail the specific situation when $j=3$.     We do so for two reasons:   this  case will provide some insight as to how the general case works, and, as it turns out, the $j=3$  case provides a sort of ``sweet spot" in the general setting.  We conclude the section by showing how Theorem \ref{Thm:SNF} dramatically simplifies the proof of the corresponding result in the $j=2$ case as compared to the proof given  in \cite{AA}.


An important role in the $j=3$ case will be played in this situation by the elements of the {\it Narayana's Cows sequence} $G$, defined recursively by setting
$$G(1) = 1, \ \  G(2) = 1, \ \  G(3) = 1, \  \mbox{ and }  \ G(n) = G(n-1) + G(n-3) \mbox{ for all $n \geq 4$}.$$
  (We may also define $G(0) = 0$, $G(-1)=0$, $G(-2)=1$ and $G(-3)=0$ consistent with the given recursion equation). This name is used in the Online Encyclopedia of Integer Sequences \cite[Sequence A000930]{OEIS}.   (Indeed, this sequence has gained some notoriety in popular culture, including the composition of a musical piece based on it.)    The first few terms of the sequence $G(n)$ ($n\geq 1$) are:
$$G: \ \ 1,1,1,2,3,4,6,9,13,19,28,41,60,88,129, \dots$$

By the aforementioned \cite[Proposition 1.3]{AA} we have that $M_{C_n^3}^*$ is indeed a group.  In  $M_{C_n^3}^*$ we have
\begin{align*}
 [v_1] & = [v_2] + [v_4] = ([v_3] + [v_5]) + [v_4] = ([v_4] + [v_6]) + [v_5] + [v_4] \\
 & = 2[v_4] + [v_5] + [v_6] = 2([v_5] + [v_7]) + [v_5] + [v_6] \\
 & = 3[v_5] + [v_6] + 2[v_7] = \cdots
\end{align*}
which by an easy induction  gives, for $1\leq i \leq n$,
$$[v_1] = G(i-1)[v_{i-1}] + G(i-3)[v_i] + G(i-2)[v_{i+1}] .$$
Thus is the Narayana's Cows sequence  related to the structure of $M_{C_n^3}^*$

%
%
%
%

Setting $i=n$, and using that $[v_{n+1}] = [v_1]$ by notational convention, we get in particular that
$[v_1] = G(n-1)[v_{n-1}] + G(n-3)[v_n] + G(n-2)[v_1]$, so that
\begin{equation*}
0 = G(n-1)[v_{n-1}] + G(n-3)[v_n] + (G(n-2)-1)[v_1]  \ \ \ \  \mbox{in} \ M_{C_n^3}^*.
\end{equation*}

\bigskip

It will be quite useful  to have an expression for the $3\times 3$ matrix $M_{3}^n$ in terms of the Narayana's Cows sequence, which is the content of the following lemma.

\begin{lemma}\label{Lem:M3n}
Let $G$ denote the Narayana's Cows sequence.  Then for each $n \in \N$,
\[
M_3^{n} =
\mat{
G(n-2) & G(n-1) & G(n)   \\
G(n-3) & G(n-2) & G(n-1) \\
G(n-1) & G(n)   & G(n+1)
} .
\]
\end{lemma}

\begin{proof} The proof is by induction.  As mentioned previously, we may extend  the $G(n)$ sequence by setting $G(0)=0$, $G(-1)=0$, and $G(-2)=1$.  Thus  we have that the statement is true for $n=1$:
\[
M_3 =
\mat{
G(-1) & G(0) & G(1)   \\
G(-2) & G(-1) & G(0) \\
G(0) & G(1)   & G(2)
}.
\]
Now suppose that 
\[
M_3^{n-1} =
\mat{
G(n-3) & G(n-2) & G(n-1) \\
G(n-4) & G(n-3) & G(n-2) \\
G(n-2) & G(n-1)   & G(n)
}
\]
Then
\begin{align*}
M_3^n = M_3^{n-1} \cdot M_3 &=
\mat{
G(n-3) & G(n-2) & G(n-1) \\
G(n-4) & G(n-3) & G(n-2) \\
G(n-2) & G(n-1) & G(n)
}
\cdot
\mat{
0 & 0 & 1 \\
1 & 0 & 0 \\
0 & 1 & 1
} \\
& =
\mat{
G(n-2) & G(n-1) & G(n-3) + G(n-1) \\
G(n-3) & G(n-2) & G(n-4) + G(n-2) \\
G(n-1) & G(n)   & G(n-2) + G(n)
}
\\
& =
\mat{
G(n-2) & G(n-1) & G(n)   \\
G(n-3) & G(n-2) & G(n-1) \\
G(n-1) & G(n)   & G(n+1)
}
\end{align*}
as desired.
\end{proof}

\begin{corollary}\label{Smithnf3}
For all $n \geq 3$,
\[
{\rm Coker}(A_{C_n^3} - I_n) \cong  {\rm Coker}
\mat{
G(n-2)-1 & G(n-3) & G(n-1) \\
G(n-1) & G(n-2)-1 & G(n) \\
G(n) & G(n-1)  & G(n+1)-1
},
\]
where $G$'s are the Narayana's Cows numbers.
\end{corollary}

\begin{proof} The result follows immediately from Theorem~\ref{Thm:SNF}
and Lemma~\ref{Lem:M3n}.
\end{proof}


We now analyze the  Smith Normal Form  for the matrix $M_3^n - I_3$.

\begin{definition}\label{DefDeterminantDivisors}
Let $n \in \mathbb{N}$.   By Lemma \ref{Lem:M3n} we have
\[
(M_3^n)^t - I_3 =
\mat{
G(n-2)-1 & G(n-3) & G(n-1) \\
G(n-1) & G(n-2)-1 & G(n) \\
G(n) & G(n-1)  & G(n+1)-1
}.
\]
For $i=1,2,3$ and each $n\geq 1$,  the corresponding $i$-th determinant divisors $ \alpha_1(n),\alpha_2(n),$ and $\alpha_3(n)$  of $(M_3^n)^t - I_3$ have the following values.



\medskip

\underline{$i=1$}:    \ $\alpha_1(n)$ is the greatest common divisor of the  nine $1\times 1$ minors $(M_3^n)^t - I_3$, i.e., of the nine entries of $(M_3^n)^t - I_3$.   By eliminating repeated entries, we see that $\alpha_1(n)$ is the greatest common divisor of five integers, to wit:
$$\alpha_1(n) = \gcd \{ \ G(n-2)-1, \ G(n-3), \ G(n-1), \ G(n), \ G(n+1)-1 \ \}. $$

\medskip

\underline{$i=2$}:   \ $\alpha_2(n)$ is the greatest common divisor of the  nine $2\times 2$ minors $(M_3^n)^t - I_3$, i.e., of the  determinants of the nine $2\times 2$ submatrices  of $(M_3^n)^t - I_3$.   By doing the standard computation for determinants of $2\times 2$ matrices, and then eliminating repeated results, we see that $\alpha_2(n)$ is the greatest common divisor of six integers, to wit:

\medskip

$\alpha_2(n) =  \ \gcd \{ \ (G(n-2)-1)^2-G(n-1)G(n-3), \  \
 (G(n-2)-1)G(n-1)-G(n)G(n-3), \ $

 \vspace{.05in}

 $\hspace{1in}
  (G(n-2)-1)(G(n+1)-1)-G(n)G(n-1), \ \  $
  $
 \ (G(n-2)-1)G(n)-G(n-1)^2,\ $

  \vspace{.05in}

 $\hspace{1in}  G(n-1)(G(n+1)-1)-G(n)^2, \
 \ G(n-3)(G(n+1)-1)-G(n-1)^2  \ \} .$

\medskip

\underline{$i=3$}:    \ $\alpha_3(n)$ is the greatest common divisor of the one $3\times 3$ minor of  $(M_3^n)^t - I_3$, in other words,
$$\alpha_3(n) =  \ |  {\rm det}((M_3^n)^t - I_3) |.$$

\end{definition}

\medskip

By Proposition \ref{Lem:H_k(n)=0}, since $3 \not\equiv 5 \ {\rm mod} 6$, we have that ${\rm Coker}(I_n - A_{C_n^j}^t)$ is finite, so that by Theorem \ref{Thm:SNF} ${\rm Coker}((M_3^n)^t - I_3)$ is also finite, which yields that  necessarily none of the entries in the Smith Normal Form of $(M_3^n)^t - I_3$ is zero.
Therefore, by the Determinant Divisors Theorem, the Smith Normal Form of the matrix $(M_3^n)^t - I_3$ is given by:
\[
{\rm SNF}((M_3^n)^t - I_3) =
\mat{
\alpha_1(n) & 0 & 0 \\
0 & \frac{\alpha_2(n)}{\alpha_1(n)} & 0 \\
0 & 0  & \frac{\alpha_3(n)}{\alpha_2(n)}
},
\]

\smallskip

\noindent
where the values of $\alpha_1(n), \alpha_2(n),$ and $ \alpha_3(n)$ are as presented in Definition \ref{DefDeterminantDivisors}.


As a consequence, Corollary \ref{Smithnf3} immediately yields the following result.

\begin{corollary}\label{Cor:K_0intermsofalpha} Let $n \in \mathbb{N}$. Then
$$K_0(L_K(C_n^3)) \cong \mathbb{Z}_{\alpha_1(n)} \oplus \mathbb{Z}_{\frac{\alpha_2(n)}{\alpha_1(n)}} \oplus \mathbb{Z}_{\frac{\alpha_3(n)}{\alpha_2(n)}}.$$
\end{corollary}
%

Corollary \ref{Cor:K_0intermsofalpha} is the key ingredient we will utilize to prove the main result about the structure of the $K_0(L_K(C_n^3))$.

\begin{remark}
{\rm Our goal for the remainder of this section will be to present a more efficient  description of the determinant divisors $\alpha_1(n), \alpha_2(n)$ and $\alpha_3(n)$ than those given in Definition \ref{DefDeterminantDivisors}.
Our motivation is as follows.   In \cite{AA}, a description of $K_0(L_K(C_n^2))$ is given in terms of greatest common divisors of pairs of integers involving terms of the Fibonacci sequence.   Our aim   here is to  establish the  analogous result, by describing $K_0(L_K(C_n^3))$ in terms of greatest common divisors of triples of integers involving terms of the Narayana's Cows sequence.

 }
\end{remark}


For integers $a,b$, $\gcd \{a,b\}$ denotes the greatest common divisor of $a$ and $b$.   A key role will be played by the following integer.

\begin{definition}\label{Def:d_3}
{\rm For any positive integer $n$ we define
$$d_3(n) := \gcd   \{ \ G(n-1), \ G(n-3),  \ G(n-2)-1 \ \}.$$

\noindent
The first few terms of the sequence $d_3(n)$ ($n\geq 1$) are:
$$d_3: \ \ 1,1,1,1,1,1,2,3,1,1,1,1,1,4,1,3,1,1, \dots $$}
\end{definition}






To begin with we show that the first determinant divisor $\alpha_1(n)$ coincides with the integer  $d_3(n)$ given in Definition \ref{Def:d_3}. To achieve it recall that we  use the following well-known property of greatest common divisors:   if $b_{n+1}$ is a $\Z$-linear combination of the integers $b_1, b_2, \dots, b_n$, then
 $$   (\ast) \ \ \ \ \ \ \ \ {\rm gcd} \{ b_1, b_2, \dots, b_n, b_{n+1} \}= {\rm gcd}\{b_1, b_2, \dots, b_n\}.$$

\begin{lemma}\label{Lem:alpha_1=d_3} Let $n \in \mathbb{N}$ and $d_3(n) = {\rm gcd} \{G(n-1),G(n-3), G(n-2)-1\}$. Then $\alpha_1(n)= d_3(n)$.
\end{lemma}

\begin{proof} According to Definition \ref{DefDeterminantDivisors},
$$\alpha_1(n) := {\rm gcd}\{G(n-2)-1,G(n-3),G(n-1),G(n),G(n+1)-1\}.$$
But   $G(n+1)-1 = G(n) + (G(n-2) - 1)$, and in turn $G(n) = G(n-3) + G(n-1)$,  so applying $(\ast)$ twice in order gives
\begin{align*}
\alpha_1(n)  & =   {\rm gcd}\{ \ G(n-2)-1, \ G(n-3), \ G(n-1), \ G(n), \ G(n+1)-1\ \} \\
 & =   {\rm gcd}\{ \ G(n-2)-1, \ G(n-3), \ G(n-1), \ G(n) \ \} \\
&  =  {\rm gcd}\{ \ G(n-2)-1, \ G(n-3), \ G(n-1) \ \}  = d_3(n). \hfill \qedhere
\end{align*}
\end{proof}

We focus now our attention in analyzing the second determinant divisor $\alpha_2(n)$ given in Definition \ref{DefDeterminantDivisors}. Thanks to property $(\ast)$ we may  reduce the number of terms appearing in its expression.

\begin{definition}{\rm Let $n \in \mathbb{N}$. We  define
\begin{align*}
d'_3(n) = {\rm gcd} & \{G(n-1)G(n-3)-(G(n-2)-1)^2, \\
& \ G(n)G(n-3)-G(n-1)(G(n-2)-1), \\
& \ G(n-1)^2-G(n)(G(n-2)-1)\}.
\end{align*}}
\end{definition}

The first terms of the $d'_3(n)$ sequence ($n\geq 1$)  are:
$$d'_3: \ \ 1,1,1,1,1,1,4,9,1,1,1,1,1,16,1,9,1,1,...$$

\begin{lemma}\label{Lem:alpha_2} Consider the second determinant divisor $\alpha_2(n)$ from Definition \ref{DefDeterminantDivisors}. Then
$$\alpha_2(n) = d'_3(n).$$
\end{lemma}

\begin{proof} By definition we have that
\begin{align*}
\alpha_2(n) := & \ {\rm gcd}\{(G(n-2)-1)^2-G(n-1)G(n-3),\\
& \ (G(n-2)-1)G(n-1)-G(n)G(n-3), \\
& \ (G(n-2)-1)(G(n+1)-1)-G(n)G(n-1),\\
& \ (G(n-2)-1)G(n)-G(n-1)^2,\\
& \ G(n-1)(G(n+1)-1)-G(n)^2, \\
& \ G(n-3)(G(n+1)-1)-G(n-1)^2 \}.
\end{align*}
Taking into account that  $G(n+1)-1 = G(n) + (G(n-2) - 1)$ and $G(n) = G(n-3) + G(n-1)$,  and applying $(\ast)$ we have
\begin{align*}
 \alpha_2(n) & =  {\rm gcd}\{(G(n-2)-1)^2-G(n-1)G(n-3),\\
& \hskip1.1cm  (G(n-2)-1)G(n-1)-G(n)G(n-3), \\
& \hskip1.1cm (G(n-2)-1)G(n)-G(n-1)^2,\\
& \hskip1.1cm G(n-1)(G(n+1)-1)-G(n)^2, \\
& \hskip1.1cm  G(n-3)(G(n+1)-1)-G(n-1)^2 \}\\
& = {\rm gcd}\{(G(n-2)-1)^2-G(n-1)G(n-3),\\
& \hskip1.1cm(G(n-2)-1)G(n-1)-G(n)G(n-3), \\
& \hskip1.1cm (G(n-2)-1)G(n)-G(n-1)^2, \\
& \hskip1.1cm G(n-3)(G(n+1)-1)-G(n-1)^2 \}\\
& = {\rm gcd}\{(G(n-2)-1)^2-G(n-1)G(n-3),\\
& \hskip1.1cm (G(n-2)-1)G(n-1)-G(n)G(n-3), \\
& \hskip1.1cm (G(n-2)-1)G(n)-G(n-1)^2 \}  \\
& = {\rm gcd}\{G(n-1)G(n-3)-(G(n-2)-1)^2,\\
& \hskip1.1cm G(n)G(n-3)-G(n-1)(G(n-2)-1), \\
& \hskip1.1cm             G(n-1)^2-G(n)(G(n-2)-1)\}\\
& \ \ \ \ \ \ \ \ \ \ \ \ \  \mbox{(since ${\rm gcd}\{a,b\}={\rm gcd}\{-a,b\}$ for any integers $a,b$) }\\
& = d'_3(n). \hfill \qedhere
\end{align*}
\end{proof}


Finally, we  show that the third determinant divisor $\alpha_3(n)$ appearing in Definition~\ref{DefDeterminantDivisors} exactly coincides with  $H_3(n)$, the $n^{\mathrm{th}}$ term of the third Haselgrove sequence.
As described above,
\[
H_k(n)
:= \left| \det(I_n - A_{C_n^k}^t) \right|
= \left| \prod_{\ell = 0}^{n-1} \left( 1 - \omega_{\ell} - \omega_{\ell}^k \right) \right|
\]
where $\omega_{\ell} = e^{\tfrac{2 \pi i \ell}{n}}$  ($0\leq \ell \leq n-1$)  are the $n$ distinct $n^{\mathrm{th}}$ roots of unity in $\mathbb{C}$.
We are particularly interested in the third Haselgrove sequence $H_3(n)$,
of which the first few terms ($n\geq 1$) are:
\[
H_3:  \ \ 1,3,1,3,11,9,8,27,37,33,67,117,131,192,341,\dots
\]
This sequence has many interesting number-theoretic characteristics (e.g., $H_3(n)$ is a divisibility sequence); we investigate this and additional properties  in \cite{AEG2}.






\begin{proposition}\label{Cor:alpha_3=H_3} Let $n \geq 1$.    Consider the $3 \times 3$ matrix
\[
(M_3^n)^t - I_3 =
\mat{
G(n-2)-1 & G(n-3) & G(n-1) \\
G(n-1) & G(n-2)-1 & G(n) \\
G(n) & G(n-1)  & G(n+1)-1
},
\]
and recall that  $\alpha_3(n) : = |{\rm det}((M_3^n)^t - I_3)|$. Then $\alpha_3(n) = H_3(n) \neq 0$.
\end{proposition}

\begin{proof}
By Proposition ~\ref{Lem:H_k(n)=0} we have $H_3(n) \neq 0$ for all $n\geq 1$. Thus we may invoke the previously cited ~\cite[Proposition 2.3]{AA} and~\cite[Proposition 2.5]{AA} to get
$
 H_3(n) = |\det(I_n - A_{C_n^3}^t)|.
$
By Theorem~\ref{Thm:SNF}, $A_{C_n^3} - I_n$ and $M_3^n - I_3$ have isomorphic cokernels. Therefore, using Proposition \ref{Prop:Smithnormalform} we get
\[
|\det(I_n - A_{C_n^3}^t)| = |\det(A_{C_n^3} - I_n)| = |\det(M_3^n - I_3)|= |\det((M_3^n)^t - I_3)| :=
\alpha_{3}(n).
\]
We conclude that $\alpha_{3}(n) = H_{3}(n)$ for all $n \geq 1$.
\end{proof}

We are now in a position to give the main result of this section. Applying Corollary \ref{Cor:K_0intermsofalpha}, Lemma \ref{Lem:alpha_1=d_3},  Lemma \ref{Lem:alpha_2} and Proposition  \ref{Cor:alpha_3=H_3} finally we get:

\begin{theorem}\label{Cor:K_0oflpasC_n^3} Let $n \in \mathbb{N}$. Then
$$K_0(L_K(C_n^3)) \cong \mathbb{Z}_{d_3(n)} \oplus \mathbb{Z}_{\frac{d'_3(n)}{d_3(n)}} \oplus \mathbb{Z}_{\frac{H_3(n)}{d'_3(n)}}.$$
\end{theorem}

\begin{example} Using Theorem \ref{Cor:K_0oflpasC_n^3}, here are  explicit descriptions of the Grothendieck groups $K_0(L_K(C_n^3))$ which arise for small values of $n$.
\begin{align*}
& n=3: \ \ K_0(L_K(C_3^3)) \cong \{0\}\\
& n=4: \ \ K_0(L_K(C_4^3)) \cong \mathbb{Z}_3\\
& n=5: \ \ K_0(L_K(C_5^3)) \cong \mathbb{Z}_{11}\\
& n=6: \ \ K_0(L_K(C_6^3)) \cong \mathbb{Z}_9 \\
& n=7: \ \ K_0(L_K(C_7^3)) \cong \mathbb{Z}_2 \oplus \mathbb{Z}_2 \oplus \mathbb{Z}_2.
\end{align*}
It is also possible for the $K_0$ group to consist of exactly two nontrivial direct summands; for instance,    $n=30$: \ $K_0(L_K(C_{30}^3)) \cong \mathbb{Z}_{31} \oplus \mathbb{Z}_{3069}$.
\end{example}
\bigskip

To finish this section we demonstrate that the results of~\cite{AA} follow from Theorem~\ref{Thm:SNF}.
When $j=2$, the companion matrix is
\[
M_2 =
\mat{
0 & 1 \\
1 & 1 \\
}.
\]
Let $F(n)$ be the $n^\mathrm{th}$ Fibonacci number. Then a well-known Fibonacci identity states that
\[
M_2^n - I_2 =
\mat{
F(n-1) - 1 &   F(n)     \\
    F(n)   & F(n+1) - 1 \\
}.
\]
So the determinant divisors will be
\[
\alpha_1 \eg \gcd(F(n-1)-1,F(n)) := d_2(n)
\]
and
\[
\alpha_2 = \det(M_2^n - I_2) = (F(n+1)-1) (F(n-1)-1) - F(n)^2.
\]
Another Fibonacci identity states that
\[
F(n+1) F(n-1) - F(n)^2 = (-1)^n.
\]
Putting these two together,
\begin{align*}
\alpha_2 &= F(n+1) F(n-1) - F(n+1) - F(n-1) + 1 - F(n)^2 \\
         &= (-1)^n - F(n+1) - F(n-1) + 1 \\
         &= -(F(n+1) + F(n-1) - 1 - (-1)^n),
\end{align*}
which is negative of the formula for $H_2(n)$ in Equation (HtoF) of~\cite[Proposition 4.4]{AA}.
Applying the Smith Normal Form, we obtain the main result of~\cite[Theorem 4.13]{AA}:
\[
K_0(L_K(C_n^2)) \cong \mathbb{Z}_{d_2(n)} \oplus \mathbb{Z}_{\frac{H_2(n)}{d_2(n)}} .
\]

\section{Leavitt path algebras of the graphs $C_n^3$}\label{Sect:K_0lpasC_n^3}

As mentioned in the introductory remarks, the primary motivation for identifying the structure of the group $M_{C_n^j}^*$ is in its realization as the Grothendieck group of the Leavitt path algebra $L_K(C_n^j)$, and subsequent utilization in identifying $L_K(C_n^j)$ up to isomorphism.   In the final section of the article we bring this program to fruition.   We begin by recalling briefly some of the basic ideas; for additional information, see e.g. \cite{AAS}.

\smallskip

{\bf Definition of Leavitt path algebra.}   Let $K$ be a field, and let $E = (E^0, E^1, r,s)$ be a directed  graph with vertex set $E^0$ and edge set $E^1$.   The {\em Leavitt path $K$-algebra} $L_K(E)$ {\em of $E$ with coefficients in $K$} is  the $K$-algebra generated by a set $\{v\mid v\in E^0\}$, together with a set of variables $\{e,e^*\mid e\in E^1\}$, which satisfy the following relations:

(V)   \ \ \  \ $vw = \delta_{v,w}v$ for all $v,w\in E^0$, \

(E1) \ \ \ $s(e)e=er(e)=e$ for all $e\in E^1$,

(E2) \ \ \ $r(e)e^*=e^*s(e)=e^*$ for all $e\in E^1$,

(CK1) \ $e^*e'=\delta _{e,e'}r(e)$ for all $e,e'\in E^1$,

(CK2)Ê\ \ $v=\sum _{\{ e\in E^1\mid s(e)=v \}}ee^*$ for every   $v\in E^0$ for which $0 < |s^{-1}(v)| < \infty$.

An alternate description of $L_K(E)$ may be given as follows.  For any graph $E$ let $\widehat{E}$ denote the ``double graph" of $E$, gotten by adding to $E$ an edge $e^*$ in a reversed direction for each edge $e\in E^1$.   Then $L_K(E)$ is the usual path algebra $K\widehat{E}$, modulo the ideal generated by the relations (CK1) and (CK2).     \hfill $\Box$

\smallskip

It is easy to show that $L_K(E)$ is unital if and only if $|E^0|$ is finite.  This is of course the case when $E = C_n^j$.    We now have the necessary background information in hand which allows us to present the powerful tool which will yield a number of key results.

\smallskip
 {\bf The Restricted Algebraic KP Theorem.} \cite[Corollary 2.7]{ALPS} Suppose $E$ and $F$ are finite graphs for which the Leavitt path algebras $L_K(E)$ and $L_K(F)$ are purely infinite simple.   Suppose that there is an isomorphism $\varphi : K_0(L_K(E)) \rightarrow K_0(L_K(F))$ for which $\varphi([L_K(E)]) = [L_K(F)]$, and suppose also that the two integers  ${\rm det}(I_{|E^0|}  -  A_E^t)$ and ${\rm det}(I_{|F^0|}-A_F^t)$ have the same sign (i.e., are either both nonnegative, or  both nonpositive).    Then $L_K(E) \cong L_K(F)$ as $K$-algebras.

\smallskip
The proof of the Restricted Algebraic KP Theorem utilizes deep results and ideas in the theory of symbolic dynamics. The letters K and P in its name derive from E. Kirchberg and N.C. Phillips, who (independently in 2000) proved an analogous result for graph $C^*$-algebras.  (We note that this analogous result does not include the hypothesis on the signs of the germane determinants; it is not known whether this hypothesis is required for the algebraic result, hence the addition of the word 'Restricted' to the name.)

We are also in position to apply Algebraic KP Theorem to explicitly realize the algebras $L_K(C_n^3)$ as the Leavitt path algebras of graphs having four vertices.   The following will be important here:  by \cite[Proposition 1.5]{AA},  for any pair $n,j$  we have that the identity element of the group $K_0(C_n^j)$ is the element $[L_K(C_n^j)] = \sum_{v\in E^0} [v]$.

\begin{proposition}\label{Propgraphfourvertices}   Let $n\in \N$. The Leavitt path algebra $L_K(C_n^3)$ is isomorphic to the Leavitt path algebra $L_K(E_n)$, where  $E_n$ is the graph with four vertices given by
$$
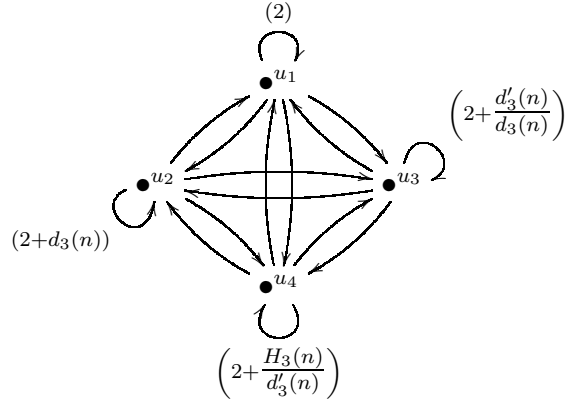
\xymatrix{ &   {\bullet}^{u_1}  \ar@(lu,ru)^{(2)} \ar@/^{5pt}/[dl]\ar@/^{5pt}/[dr] \ar@/^{5pt}/[dd] &  \\
           {\bullet}^{u_2} \ar@/^{5pt}/[ur] \ar@/^{5pt}/[dr] \ar@/^{5pt}/[rr] \ar@(l,d)_{(2+d_3(n))}&   & {\bullet}^{u_3} \ar@/^{5pt}/[ll] \ar@/^{5pt}/[dl] \ar@/^{5pt}/[ul] \ar@(u,r)^{\left(2+\tfrac{d'_3(n)}{d_3(n)}\right)} \\
           &   {\bullet}^{u_4} \ar@/^{5pt}/[ur] \ar@/^{5pt}/[ul] \ar@/^{5pt}/[uu] \ar@(rd,ld)^{\left(2+\tfrac{H_3(n)}{d'_3(n)}\right)} &
}$$ (where the numbers in parentheses indicate the number of loops at the indicated  vertex).
\end{proposition}

\begin{proof}
Using the characterization given in \cite{AAP2}, it follows easily that the graph $E_n$ satisfies the conditions for $L_K(E_n)$ to be unital purely infinite simple.    The incidence matrix of $E_n$ is $$A_{E_n}=\left(\begin{matrix} 2 & 1 & 1 & 1\\ 1 & 2+d_3(n) & 1 & 1\\ 1 & 1 & 2+\tfrac{d'_3(n)}{d_3(n)} & 1\\ 1 & 1 & 1 & 2+\tfrac{H_3(n)}{d'_3(n)} \end{matrix}\right),$$ so that $$I_4-A^t_{E_n}=-\left(\begin{matrix} 1 & 1 & 1 & 1\\ 1 & 1+d_3(n) & 1 & 1\\ 1 & 1 & 1+\tfrac{d'_3(n)}{d_3(n)} & 1\\ 1 & 1 & 1 & 1+\tfrac{H_3(n)}{d'_3(n)} \end{matrix}\right).$$
\noindent
A straightforward computation yields  that the Smith Normal Form of $I_4-A^t_{E_n}$ is  $$\left(\begin{matrix} 1 & 0 & 0 & 0\\ 0 & d_3(n) & 0 & 0\\ 0 & 0 & \tfrac{d'_3(n)}{d_3(n)} & 0\\ 0 & 0 & 0 & \tfrac{H_3(n)}{d_3(n)}\end{matrix}\right),$$ which immediately yields that  $K_0(L_K(E_n))$ is isomorphic to $\Z_{d_3(n)}\oplus  \Z_{\frac{d'_3(n)}{d_3(n)}} \oplus \Z_{\frac{H_3(n)}{d'_3(n)}}$.

Also, it is straightforward to check that $\det(I_4-A^t_{E_n})=-H_3(n)<0$.  (We remind the reader that the sign of the determinant of a matrix cannot be gleaned from the Smith Normal Form of the matrix; specifically, one must compute the determinant of $I_4-A^t_{E_n}$ directly from the matrix itself.)

Finally,  by invoking the relation in $K_0(L_K(E_n))$ at  $u_1$, we have
\begin{align*}
 [u_1] + [u_2] + [u_3] + [u_4]  & = (2[u_1] + [u_2] + [u_3] + [u_4]) + [u_2] + [u_3] + [u_4] \\
& = 2([u_1] + [u_2] + [u_3]+[u_4]),
\end{align*}
so that $\sigma = [u_1] + [u_2] + [u_3] +[u_4]$ satisfies $\sigma =2\sigma$ in the group $K_0(L_K(E_n))$, which gives  that $\sigma = [u_1] + [u_2] + [u_3]+[u_4]$ is the identity element of $K_0(L_K(E_n))$.

So   the purely infinite simple unital Leavitt path algebras  $L_K(C_n^3)$ and $L_K(E_n)$ have these properties:

\begin{enumerate}
\item  $K_0(L_K(C_n^3)) \cong K_0(L_K(E_n))$ (as each is isomorphic to $\Z_{d_3(n)}\oplus  \Z_{\frac{d'_3(n)}{d_3(n)}} \oplus \Z_{\frac{H_3(n)}{d'_3(n)}}$),

\item  this  isomorphism necessarily takes $[L_K(C_n^3)]$ to $[L_K(E_n)]$ (as each of these is the identity element in their respective $K_0$ groups), and

\item   both  $\det(I_n-A^t_{C_n^3})$   and   $\det(I_4-A^t_{E_n})$ are negative.

\end{enumerate}

Thus  the graphs $C_n^3$ and $E_n$ satisfy the hypotheses of the Restricted Algebraic KP Theorem, and so the desired isomorphism $L_K(C_n^3) \cong L_K(E_n)$ follows.
\end{proof}

\begin{remark}
{\rm   Although it is relatively easy  to produce graphs $F_n$ having three vertices for which $K_0(L_K(F_n))\cong K_0(L_K(C_n^3))$, we do not know how to produce such graphs for which $[F_n]$ is the zero element in $K_0(L_K(F_n))$, which therefore precludes  us from applying the Restricted Algebraic KP Theorem to the Leavitt path algebras of these graphs. }
\end{remark}

\begin{remark}
{\rm   Using the template afforded by the $4$-vertex graph  presented in Proposition \ref{Propgraphfourvertices},  for each pair $j,n$ with $0\leq j \leq n-1$  one can easily construct a graph $E_n(j)$ having $j+1$ vertices, for which the Leavitt path algebras $L_K(E_n(j))$ and $L_K(C_n^j)$ are isomorphic.
}
\end{remark}

A number of intriguing number-theoretic properties of the Haselgrove sequences and group-theoretic properties of the groups $K_0(L_K(C_n^j))$ arose in the context of the investigation presented in this article.  For instance, as mentioned previously, the $H_k(n)$ sequences can be shown to be divisibility sequences.    For another example, in the $j=3$ case we may give a more explicit description of the integers $d_3^\prime(n)$, as a product  of a power of $d_3(n)$ with an ``indicator factor".    (In retrospect,  we see that an analogous statement arose in the proof of the corresponding result in  the $j=2$ case carried out in \cite{AA}, but this indicator factor  turned out  not to  play a role  in the invariant factors representation of the abelian group $K_0(L_K(C_n^2))$.)    However, such an ``indicator factor"  description does not extend to the cases $j\geq 4$; for this reason we refer to the $j=3$ case as a ``sweet spot" in this setting.
These and many additional properties will be presented in \cite{AEG2}.

\section*{acknowledgements}
The authors would like to thank G. Aranda Pino and M. Iovanov for fruitful discussions during the preparation of this paper. Some of these results were anticipated and suggested by looking at output from the software package \emph{Magma}. The authors are grateful to A. Viruel for his valuable help with this software.

The first author was partially supported by a Simons Foundation Collaboration Grant \#208941.   The third author was partially supported by the Spanish MEC and Fondos FEDER through projects MTM2013-41208-P and MTM2016-76327-C3-1-P; by the
Junta de Andaluc\'{\i}a and Fondos FEDER, jointly, through project FQM-7156; and by the grant "Ayudas para la realizaci\'on de estancias en centros de investigaci\'on de calidad" of the "Plan Propio de Investigaci\'on y Transferencia" of the University of M\'alaga, Spain. Part of this work was carried out
during a visit of the third author to the University of Colorado, Colorado Springs, USA. The third author thanks this host institution for its warm hospitality and support.

\end{document}